\documentclass[12pt]{amsart}

\usepackage{graphicx}
\usepackage{hyperref}
\usepackage{enumitem}
\usepackage{comment}
\usepackage{color}
\usepackage{amsmath,amssymb}
\usepackage{amsfonts}
\usepackage{enumitem}
\usepackage{dsfont}
\usepackage{tikz}
\usetikzlibrary{patterns, positioning}
\usepackage[T1]{fontenc}
\usepackage[ansinew]{inputenc}
\usepackage{lmodern}
\usetikzlibrary{arrows}

\newcommand{\der}{\operatorname d\hspace{-0.1em}}
\newcommand\R{\mathbb R}

\newtheorem{theorem}{Theorem}[section]
\newtheorem{corollary}[theorem]{Corollary}
\newtheorem{lemma}[theorem]{Lemma}
\newtheorem{proposition}[theorem]{Proposition}
\newtheorem*{corollary*}{Corollary \ref{intro}}
 \theoremstyle{definition}
\newtheorem{definition}[theorem]{Definition}

 \theoremstyle{remark}
\newtheorem{remark}[theorem]{Remark}
\newtheorem*{notation*}{Notation}

\numberwithin{equation}{section}

\begin{document}

\title{Removable Singularities of  $m$-Hessian Equations}

\author{H\"ulya Car}
\address{Universit\"at Ulm, Helmholtzstrasse 18, 89081 Ulm, Germany}
\email{huelya.arslan@uni-ulm.de}

\author{Ren\'e Pr\"opper}
\email{proepper@zedat.fu-berlin.de}
\thanks{}


\date{}

\begin{abstract}
In this paper we give a new, less restrictive condition for removability of singular sets, $E$, of smooth solutions to  the m-Hessian equation (and also for more general fully nonlinear elliptic equations)  in $\Omega \setminus E$, $\Omega \subset \mathbb R^n$. Besides the existence and regularity results for these equations, the proof only makes use of the classical elliptic theory, i.e. the classical maximum principles and a Hopf lemma.
\end{abstract}

\maketitle

\section{Introduction} \label{introduction}
We are primarily concerned with the question of removability of  singular sets of smooth solutions for elliptic $m$-Hessian equations which are of the form  
\begin{equation} \label{nonsing}
               F_m(D^2 u)=f(x) \text{ in } \Omega,  \quad u=\varphi \text{ on } \partial \Omega,
              \end{equation}
or in the setting with singularities
\begin{equation} \label{Hessian}
     F_m(D^2u) =f(x)  \quad \text{ in } \Omega\setminus E, \quad u=\varphi \text{ on } \partial \Omega,
\end{equation}
where $m\le n$, $\Omega$ is a domain of $\mathbb R^n$, $f$ a sufficient smooth, positive function in $\Omega$, $\varphi$ a sufficient smooth function on $\partial \Omega$, and $E$ the singular set of the solution $u$. Of course, we have to impose certain conditions on the singular set $E$ and also on the solution $u$ in order to ensure removability of $E$. We describe these in detail in Section \ref{removable}.\\
The  $m$-Hessian operator $F_{m}$ is defined by
\[
F_{m}(D^{2}u)=:S_{m}(\lambda)^{1/m}:=\big (\sum\nolimits_{1\leq i_1 < \dots < i_m \leq n} \lambda_{i_1} \cdots \lambda_{i_m} \big )^{1/m},
\]
where $\lambda$ stands for the vector of eigenvalues $\lambda_{1},\dots,\lambda_{n}$ of the Hessian matrix, $D^{2}u$. They are special cases of equations of Hessian type, as considered e.g. in \cite{CNS}. \\
For $m=1$ we get the Laplace operator, but for $m\ge 2$ the $m$-Hessian equation becomes a fully nonlinear partial differential equation which is elliptic if $u$ is $m$-admissible, i.e. $S_{k}(\lambda(D^{2}u)) >0$ for each $k=1,\dots,m$.  A very important special case is the Monge-Amp\`ere equation, i.e. $m=n$, see e.g. \cite{TW2}.\\
The necessary definitions and results about $m$-Hessian equations will be given in the next section, see e.g. \cite{CNS},  \cite{W} and \cite{U}.\\
Before stating our contribution to the question of removability, we briefly want to mention the history of the problem.
In \cite{J} (1955) K. J\"orgens  proved that an isolated singular point, $p$, of a $C^2$-solution, $u$, of the Monge-Amp\`ere equation $\det (D^2 u)=1$ in two dimensions is removable if $u$ is continuously differentiable along one line through $p$. Later in \cite{SW} (1995)  F. Schulz and L. Wang generalized this to higher dimensions under the same constraints using classical elliptic theory, and R. Beyerstedt in \cite{B2} (1995) to more general elliptic Monge-Amp\`ere equations $\det(D^2 u)=\psi(x,u,Du)$ using the Aleksandrov maximum principle.\\
Exploiting a notion of capacity and weak solution the results in \cite{L} (2002) by D. Labutin and \cite{TW2} (2002)  by N. S. Trudinger and X. Wang make it plausible that for $m \leq n/2$ and if the $(n-2m)$-dimensional Hausdorff measure of $E$ is finite there should be no additional constraints on the solution necessary to guarantee removability. But, at least for the Monge-Amp\`ere equation in $\R^2$, it is known that even a point singularity of a continuous solution to $\eqref{Hessian}$ does not need to be removable without further assumptions, see \cite{J} and \cite{B1}.\\
Another way to ensure removability can be found in \cite{WZ} (1999) for $F(D^2u)=f(x)$ and later in \cite{S} (2007) F. Schulz was able to handle more general equations of the form $F(Du,D^2u)=\psi(x,u,Du)$ including the Monge-Amp\`ere, the Hessian and the Weingarten equations by taking up the idea of \cite{B2}. In both, more general elliptic equations are considered, and the proofs rely on versions of the Aleksandrov maximum principle. Hence, the conditions involve in a way the Lebegues measure of the union over the image of the lower normal mapping of the function $u-v$ in the singular set, where $u$ is the singular solution and $v$ a classical solution of \eqref{Hessian} on $\Omega$. From their condition L. Wang and N. Zhu derive in  \cite{WZ} (stated here for the case of $m$-Hessian equations):\\[1mm]
{\em Let $u\in C^{2}(\Omega\setminus E)\cap \operatorname{Lip}(\Omega)\cap                                                 C(\overline\Omega)$ be an $m$-admissible solution of equation \eqref{Hessian} such that the Dirichlet problem \eqref{nonsing} has an $m$-admissible solution $v\in C^2(\Omega)\cap C(\overline \Omega)$ with $\varphi=u_{|\partial \Omega}$. Let $E\subset\subset \Omega$ be a measurable set of dimension $l<n$. Then $E$ is removable, i.e. $(v=)u\in C^2(\Omega)$, if for every $x\in E$ there are $l+1$ independent $C^{2}$-curves $\{r_{xi}\}$ through x, $i\in \{1,2,\dots,l+1\}$, such that $u(r_{xi})\in C^{1}$}.\\[1mm]
In this paper we are able to weaken their assumptions partly and prove the following. 
\begin{corollary*}
 Let $u\in C^{2}(\Omega\setminus E) \cap  C( \overline \Omega)$ be an $m$-admissible solution of equation \eqref{Hessian} such that the Dirichlet problem \eqref{nonsing}
 has an $m$-admissible solution $v\in C^2(\Omega)\cap C(\overline \Omega)$ with $\varphi=u_{|\partial \Omega}$. Let $E \subset \subset \Omega$ be a closed  subset of a $C^{1,1}$-submanifold, $M$, of dimension $l<n$. Then $E$ is removable, i.e. $(v=)u\in C^2(\Omega)$, if for every $x \in E$ there is one continuous curve $\gamma$ through $x$, which is differentiable at $x$ and transversal (i.e. not tangential) to $M$ at $x$, such that $u$ is differentiable at $x$ along $\gamma$. 
\end{corollary*}
In fact we achieve more, see our main Theorem \ref{main}, where $E$ is not necessarily a subset of a submanifold, only relatively compact in $\Omega$,  and $\gamma$ not even a continuous curve, see Section 3 for details. We also demonstrate how our method could be applied to other elliptic equations, see Theorem \ref{generalmain}.\\
We use the idea of F. Schulz and L. Wang from \cite{SW} and expand it according to the goals of this paper. In contrast to \cite{WZ} we utilize, besides the existence and regularity results for $m$-Hessian equations (see \cite{CNS} and \cite{W}), only classical linear elliptic theory, i.e. a generalized Hopf lemma (see Appendix \ref{Hopf}), and the classical maximum principles. But here as well as in \cite{WZ} and \cite{S} the final aim is to show that the singular solution agrees with the classical solution of the corresponding Dirichlet problem in $\Omega$.

The paper is structured as follows. In {\em Section 2} we summarize some known results about $m$-Hessian equations and provide important comparison results in  {\em Subsection \ref{Comparison} } which will be used in the proof of our main theorem.\\
In {\em Section 3} we first introduce a few definitions which will help us formulate our main result (Theorem \ref{main}) in a concise and general way. Finally, in {\em Subsection \ref{general}} we briefly point out what is needed for a general elliptic equation $F(x,u, \nabla u, D^2 u)=0$ in order to carry our approach through.\\
In {\em Appendix A} we give a generalization of the usual Hopf lemma, see e.g. \cite{GT} Lemma $3.4$, because it is fundamental for our approach. It seems to be well known, see the remark in \cite{GT} after Lemma $3.4$. But we could not spot a direct proof in the literature, hence, we would like to shortly present one. For the convenience of the reader, we also cite the comparison and strong maximum principle, which are used throughout the paper.\\
In order to deduce Corollary \ref{intro} from the main theorem we need to establish the existence of a touching ball with prescribed normal unit vector to every point of a $C^{1,1}$-submanifold. This is done in {\em Appendix B}.  Note that this is a generalization of the (interior/exterior) sphere condition.
\begin{notation*}
We adopt in this paper the convention that over doubly occurring indices a summation is understood.\\
For a matrix $(a_{ij})_{i=1,\dots, n}^{ j=1,\dots, n}$ we just write $(a_{ij})$ if there should be no danger of confusion. So $(u_{ij})=D^{2}u$ denotes the Hessian matrix of $u$. \\
$B_r(p)$ will always stand for a ball with center $p$ and radius $r$.\\ 
If $\Omega' \subset \Omega$ and $u$ is a function of $\Omega$ we write $u_{|\Omega'}$ for the restriction  of $u$ to $\Omega'$. Furthermore, a domain, $\Omega$, is assumed to be an open, connected, and bounded subset of $\R^n$.
\end{notation*}

\section{$m-$Hessian equations} \label{Hesseq}

Now we give the relevant definitions and facts, which we need later. These can be found e.g. in \cite{CNS},  \cite{W} and \cite{U}.\\
Let $\Omega \subset \mathbb R^n$ be a domain and $u \in C^2(\Omega)$, then the $m$-Hessian operator ($m=1,\dots,n$) acts on $u$ as
\[
F_{m}(D^{2}u)=:S_{m}(\lambda(D^{2}u))^{1/m}:=\big (\sum\nolimits_{1\leq i_1 < \dots < i_m \leq n} \lambda_{i_1} \cdots \lambda_{i_m} \big )^{1/m},
\]
where $\lambda(D^{2}u)=(\lambda_1,\dots,\lambda_n)$ is the vector of eigenvalues of the Hessian matrix, $D^{2}u$, and $S_m$ is the $m$-th elementary symmetric function. $S_{m}(\lambda(D^{2}u))$ can also be written as the sum of the $m\times m$-principal minors of $D^2u$. $F_1$ is the Laplace operator, but for $m\ge 2$ the operator becomes fully nonlinear with the Monge-Amp\`ere case $\det( D^2 u)$, i.e. $m=n$, as the most prominent example.\\
We call a function $u \in C^2(\Omega)$ {\em $m$-admissible} if $S_{k}(\lambda(D^{2}u)) >0$ for each $k=1,\dots,m$. Of course, an $m$-admissible function is $k$-admissible for $k \le m$, where $1$-admissible agrees with strictly subharmonic and $n$-admissible with strictly convex. It follows immediately from the definition that the right hand side of \eqref{Hessian} must be positive for $m$-admissible solutions.\\
The following two properties of $m$-admissible functions  will be important for us (see \cite{CNS} Section 1, especially Proposition 1.1, and Section 3, or \cite{W} Sections 2.1 and 2.2. Note that in the definition of $m$-admissible in \cite{W} the inequality is not strict, so $m$-admissible corresponds only to degenerate elliptic there):
\begin{lemma}\label{convexelliptic}
The set of $m$-admissible functions constitutes an open convex cone in $C^2(\Omega)$. Furthermore, 
$F_m(D^2u)$ is elliptic whenever $u$ is $m$-admissible, i.e. for each $\xi\in \mathbb R^n\setminus \{0\}$ we have $\frac{\partial F_m(D^2u)}{\partial u_{ij}}\xi^{i}\xi^{j}> 0$ in $\Omega$.
\end{lemma}
A $C^{2}$-domain $\Omega$ is called {\em$(m-1)$-convex} if its boundary satisfies the condition
$S_{m-1}(\kappa)\ge c_{0}>0$  on  $\partial\Omega$
for some positive constant $c_{0}$, where $\kappa=(\kappa_{1},\dots,\kappa_{n-1})$ denotes the principle curvatures of $\partial\Omega$ with respect to its inner normal. We want to explicitly mention that every ball is $(m-1)$-convex for every $m$.\\
We will also make heavy use of the following existence result, see \cite{W} Theorem 3.4 and \cite{T}.

\begin{theorem} \label{subsec: cns} 
Let $\Omega$ be an $(m-1)$-convex domain in $\mathbb R^n$, $\partial{\Omega}\in C^{3,1}$, $f \in C^{1,1}(\overline{\Omega})$ with $f(x) \geq f_0 > 0$ and $\varphi\in C^{3,1}(\partial{\Omega})$. Then there exists a unique $m$-admissible function $u\in C^{3,\alpha}(\overline{\Omega})$, $\alpha\in (0,1)$, solving the Dirichlet problem
\[
(DP) \begin{cases}
F_m(D^2u)= f & \text{ in } \Omega\\
u=\varphi & \text{ on } \partial{\Omega}. 
\end{cases}
\]
\end{theorem}

\begin{remark} \label{uniformly}
The regularity theory for elliptic equations gives better regularity, i.e. $u \in C^{k+2,\alpha}$, if $f \in C^{k,\alpha}$, $k \geq 1$ and $0<\alpha<1$, as this holds for solutions of second order equations, $F(x,u,\nabla u, D^2 u)=0$, in general, if  $F$ is elliptic with regard to $u$, see \cite{GT} Lemma $17.16$.\\
As mentioned in \cite{W} at the end of Section $3.1$, $F_m$ is automatically uniformly elliptic with regard to this solution.\\
It is also known that the regularity assumptions can be reduced for the Monge-Amp\`ere equation, see \cite{TW2} Theorem $4.1.$ 
\end{remark}

\subsection{Comparison lemmata} \label{Comparison} 

Let us consider for the moment a general function 
\begin{equation} \label{generalfu}
 F(x,z,p,r) \in C^1(\Omega,\mathbb R, \mathbb R^n, \mathbb R^{n \times n})
\end{equation} 
and the differential equation
\begin{equation}\label{generaleq}
  F[u]=F(x,u,\nabla u, D^2 u)=0, \quad x \in \Omega. 
\end{equation}  
Assume $v_0,\, v_1 \in C^2(\Omega)$ are solutions of \eqref{generaleq} (or e.g. $v_1$ a subsolution, i.e. $F[v_1]\geq 0$).  We define $w=v_{1}-v_{0}$ and $w_\theta = \theta v_1 + (1-\theta)  v_0$, $\theta\in [0,1]$. Then $w$ is a solution (subsolution) of the following linear equation (as in the proof of Theorem $17.1$ from \cite{GT}):
\begin{equation}\label{linearoperator}
L(v_{1},v_{0})w=Lw=  a_{ij}(x)  \frac{\partial^2w}{\partial x_i \partial x_j}+ b_{i}(x)\frac{\partial w}{\partial x_{i}}+c(x)w=0(\geq 0) 
\end{equation}
with
\[ 
a_{ij} := \int_0^1  \frac{\partial F[w_\theta]}{\partial r_{ij}} \der \theta, \;  b_{i}:=\int_0^1\frac{\partial 
F[w_\theta]}{\partial p_{i}} \der \theta,  \;  c :=\int_0^1\frac{\partial 
F[w_\theta]}{\partial z}\der \theta.
\]
 Let $\lambda(x)$ be the least and $\Lambda(x)$  the greatest eigenvalue of the symmetric matrix $(a_{ij}(x))$. These depend continuously on $x\in \Omega$ if $(a_{ij}(x))$ does (see e.g. \cite{K} Chapter 2, Sections 5.2 and 5.7).\\ 
We also write
\[
a^\theta_{ij}(x) := \frac{\partial F[w_\theta]}{\partial r_{ij}},   \quad  b^\theta_{i}(x):=\frac{\partial 
F[w_\theta]}{\partial p_{i}},    \quad c^\theta(x):=\frac{\partial 
F[w_\theta]}{\partial z}.  
\]
We call the linear operator $L=L(v_1,v_0)=L(F;v_1,v_0)$ the linearization of $v_1$ and $v_0$.

\begin{proposition}\label{linearelliptic}
Let $\mathcal U \subset C^2(\Omega)$  be a convex set such that $F$ is elliptic for all elements of  $\mathcal U$. Let $v_0,v_1 \in \mathcal U$ and $L(v_{1},v_{0})$ as in (\ref{linearoperator}). Then the following holds true:
\begin{enumerate}[label=(\roman*)]
\item $L(v_1,v_0)$ is elliptic in $\Omega$.
\item Let $\Omega' \subset \subset \Omega$. Then the coefficients $a_{ij}$, $b_i$ and $c$ are bounded in $\Omega'$. Furthermore,  $L(v_1,v_0)$ is uniformly elliptic in $\Omega'$, i.e. there are constants $\lambda_0,\Lambda_0$ with $0<\lambda_0\le \Lambda_0<\infty$ such that $ \lambda_0 \leq \lambda(x) \leq \Lambda(x) \leq \Lambda_0$ for every $x \in \Omega'$. In other words, $L(v_1,v_0)$ is locally uniformly elliptic and its coefficients are locally bounded in $\Omega$.
\item If $F \in C^1(\overline \Omega,\mathbb R, \mathbb R^n, \mathbb R^{n \times n})$ and $v_0,v_1 \in C^2(\overline \Omega)$ and $a^1_{ij}$ or $a^0_{ij}$ is positive definite in $\overline \Omega$, i.e. $F$ is uniformly elliptic in $\Omega$ with regard to $v_1$ or $v_0$, then $L$ is uniformly elliptic in $ \Omega$ and its coefficients are bounded. 
\item If in addition $\frac{\partial F(x,z,p,r)}{\partial z}=:F_z \leq 0$, then $c \leq 0$.
\end{enumerate} 
\end{proposition}

\begin{proof}
$(i)$: Due to the convexity of $\mathcal U$, $F$ is elliptic for every $w_\theta$, that is every $a^\theta_{ij}$ is positive definite. Hence, $L(v_1,v_0)$ is elliptic.\\
$(ii)$: As $F \in C^1(\Omega,\mathbb R, \mathbb R^n, \mathbb R^{n \times n})$ and $v_0,v_1 \in C^2(\Omega)$  all functions $a_{ij}^\theta(x)$, $b_i^\theta(x)$, and $c^\theta(x)$ depend for $\theta \in [0,1]$ equicontinuously on $x$. Hence, $a_{ij}(x)$, $b_i(x)$, and $c(x)$ depend continuously on $x$. This implies the boundedness of the coefficients as well as of $\lambda(x)$ and $\Lambda(x)$ on $\overline \Omega'$. Moreover, $\lambda(x)> 0$ on  $\overline \Omega'$, so $\lambda(x)\geq \lambda_0 >0$.\\  
$(iii)$: Observe first that for every $x \in \overline \Omega$ $a^\theta_{ij}(x)$ is positive semi-definite and, say, $a^1_{ij}(x)$ positive definite. Furthermore, $a^\theta_{ij}(x)$ depends continuously on $\theta$, hence, $a_{ij}(x)$ is positive definite. So the same argument as in $(ii)$ finishes the proof.\\
$(iv)$: This follows immediately from the definition of $c$.  
\end{proof}

\begin{corollary}\label{hesseprop}
The Proposition is, in particular, applicable to the linearization 
\[
L(v_1,v_0) w= \int_0^1  \frac{\partial F_m[w_\theta]}{\partial r_{ij}} \der \theta \frac{\partial^2w}{\partial x_i \partial x_j}= a_{ij}\frac{\partial^2w}{\partial x_i \partial x_j}
\]
of two $m$-admissible solutions, $v_1$ and $v_0$, of the $m$-Hessian equation \eqref{nonsing}.
\end{corollary}

\begin{proof}
This follows from the above Proposition and Lemma \ref{convexelliptic}. 
\end{proof}

\begin{lemma} \label{satz:zwischenfunktion}
Let $\Omega\subset \mathbb R^n$ be a domain, $a$ and $x_{1}$ boundary points of $\Omega$, $x_1\ne a$. Suppose $v_{1}$ and $v_{0}$ are functions such that $v_1\in C^0(\overline \Omega)$ and $v_0\in C^{k,\alpha}(\overline \Omega)$ with $k\in \mathbb N$ and $0\le\alpha\le1$. Furthermore $v_1\le v_0$ in $\overline\Omega$, $v_1(a)=v_0(a)$ and $v_1(x_1)<v_0(x_1)$. Then there exists an intermediate function $\varphi_2 \in C^{k,\alpha}(\overline \Omega)$ satisfying
\begin{equation} \label{intermediate}
 \begin{cases}
 v_1=\varphi_2=v_0 \text{ at } a\\
  v_1<\varphi_2<v_0 \text{ at } x_{1}\\
  v_1 \le \varphi_2 \le v_0 \text{ in } \overline\Omega.
 \end{cases}
\end{equation}
\end{lemma}

\begin{proof}
Because of $v_0-v_1\in C^0(\overline \Omega)$ and $(v_0-v_1)(x_1)> 0$ there exists an $\varepsilon>0$ and a closed ball, $\overline B_{\varepsilon}(x_{1})$, around $x_1$, such that $v_{0}-v_{1}>0$ in $\overline{B_{\varepsilon}}(x_{1}) \cap \overline \Omega $ and $a$ lies in its complement $\overline{B_{\varepsilon}}(x_1)^c$. \\
The minimum \[
  \min_{\overline {B_{\varepsilon}}(x_{1}) \cap \overline \Omega}(v_0-v_1)(x):=\delta'' >0
 \] 
exists. Define $\delta':=\frac{1}{2}\delta''$ and the intermediate function
\[
    \varphi_2:=v_0-\delta' \psi(x),
\]
where $\psi(x) \in C^{\infty}(\overline\Omega)$ denotes a smooth function 
in $\overline\Omega$ such that
\[
 \psi(x):=\begin{cases}
           1  \quad\text{ at } x_{1},\\
           \in [0,1]  \quad\text{ for } 0<|x-x_1|<\varepsilon,\\
           0  \quad\text{ for } |x-x_1|\ge \varepsilon.
          \end{cases}
\]
This implies $\delta' \psi\in C^{\infty}(\overline\Omega)$ and $\varphi_2 \in C^{k,\alpha}(\overline \Omega)$. Moreover, $\varphi_2 $  satisfies \eqref{intermediate} by construction. 
\end{proof}

The lemma below is an adaption to our more general situation of two lemmata from \cite{SW}.

\begin{lemma} \label{comparison}
Given a ball $B \subset \R^n$ and two $m$-admissible solutions, $v_0, v_1 \in C^0(\overline B)\cap C^2(B)$, of the $m$-Hessian equation
\[
 F_{m}(D^{2}u)=f(x) \geq f_0 >0, \, f \in C^{1,1}(\overline B),\, x \in B.
\]
Furthermore, we assume $v_0 \ge v_1$ on $\partial B$ and that there exist two boundary points  $a, x_1 \in \partial B$ with $v_0(a)=v_1(a)$ and $v_0(x_1) > v_1(x_1)$.\\
Then the following comparison results hold true (note that the case $+\infty$ is allowed):
\begin{enumerate}[label=(\roman*)]
\item $v_0(x) \geq v_1(x)$ for all $x \in B$ and, therefore,
\[
\liminf_{x \to a,\, x\in B}\frac{(v_{0}-v_{1})(x)}{\Vert x-a\Vert} \geq 0.
\]
\item If  $v_0, v_1 \in C^2(\overline B)$ and $F_m$ is uniformly elliptic in $B$ with regard to, say, $v_1$, then 
\[
 \liminf_{x\to a,\, x\in K(a)\cap B}\frac{(v_{0}-v_{1})(x)}{\Vert x-a\Vert} >0 
 \]
for every closed convex cone, $K(a)$, with apex $a$ and such that $K(a) \cap B_{\varepsilon}(a) \subset B$ for $\varepsilon>0$ small enough (see also Appendix \ref{hopf}).
\item  If $v_1\in C^0(\overline B)\cap C^2(B)$, $v_0\in C^{3,1}(\overline B)$, then it also follows for every cone $K(a)$ as above that 
\[
 \liminf_{x\to a,\, x\in K(a)\cap B}\frac{(v_{0}-v_{1})(x)}{\Vert x-a\Vert} >0. 
 \]
 \end{enumerate}
\end{lemma}

\begin{proof}
In this proof let $w:=v_1-v_0$ as before. Then, $w(a)=0$, $w(x_1)<0$ and $w \leq 0$ on $\partial B$ by assumption.\\
$(i)$: The linearization $L=L(v_{1},v_{0})$ is, due to Proposition \ref{linearelliptic} $(i)$ and Corollary \ref{hesseprop}, elliptic in $B$ and $Lw=0$ by \eqref{linearoperator}. The comparison principle (Theorem \ref{weak}, here $b \equiv c \equiv 0$) gives us $v_1-v_0=w \leq 0$ in $B$.\\ 
$(ii)$: In this case, Proposition  \ref{linearelliptic} $(iii)$ and Corollary \ref{hesseprop} say that $L(v_{1},v_{0})$ is uniformly elliptic in $B$. By assumption $w(a)=0$, and the strong maximum principle (Theorem \ref{strong}, here $b \equiv c \equiv 0$) yields because of $w\leq 0$ on $\partial B$ and $w(x_1)<0$ that $w=v_1-v_0 <0$ in $B$, so we can invoke the Hopf lemma, Lemma \ref{hopf} and get 
\[
\liminf_{x\to a,\, x\in K(a)\cap B} \frac{w(a) -w(x)}{\| x-a\|}>0 \quad \text{or} \quad \liminf_{x\to a,\, x\in K(a)\cap B}\frac{(v_0-v_1)(x)}{\Vert x-a\Vert} >0. 
\]
$(iii)$: We first observe that $F_m$ is uniformly elliptic in $B$ with regard to $v_0$, see Remark \ref{uniformly}.\\
Since we know by $(i)$ that $v_1 \leq v_0$ in $B$, we can construct, according to Lemma \ref{satz:zwischenfunktion}, an intermediate function $\varphi_2 \in C^{3,1}(\overline B)$ with

\begin{equation} \label{phi2}
 \begin{cases}
 v_1=\varphi_2=v_0 \text{ at } a\\
  v_1<\varphi_2<v_0 \text{ at } x_{1}\\
  v_1 \le \varphi_2 \le v_0 \text{ on } \partial B.
 \end{cases}
 \end{equation} 

By dint of the existence Theorem \ref{subsec: cns}, there exists an $m$-admissible function $u_2\in C^{3,\alpha}(\overline B)$ solving the Dirichlet problem
 \[
   F_m(D^2u_2)= f(x)  \text{ in } B \text{ and }
   u_2=\varphi_2  \text{ on } \partial{B}.
\]
This allows replacing $\varphi_2$ by $u_2$ in $\eqref{phi2}$. Therefore, $v_1$ and $u_2$ satisfy the requirements of item $(i)$, whereas $u_2$ and $v_0$ those of item $(ii)$. We get
\begin{align*}
 \liminf_{x\to a,\, x\in K(a)\cap B}\frac{(v_{0}-v_{1})(x)}{\Vert x-a\Vert} & \\
&\hspace{-3,7cm}\geq  \liminf_{x\to a,\, x\in K(a)\cap B}\frac{(v_{0}-u_{2})(x)}{\Vert x-a\Vert} +  \liminf_{x\to a,\, x\in K(a)\cap B}\frac{(u_{2}-v_{1})(x)}{\Vert x-a\Vert} >0.  \qedhere
\end{align*}
\end{proof}

\section{Removability of singularities} \label{removable}

First, we need some definitions to formulate our conditions appropriately.

\begin{definition}\label{admissible}
Given a domain $\Omega \subset \R^n$ and a relatively closed subset $E \subset \Omega$. Set $\mathcal A:=\{U|\, U \text{ is a connected component of } \Omega \setminus E\}$.\\
We define inductively the following sets:
\begin{align*}
 E_0&:= \partial \Omega, \\
 A_1&:=\{U \in \mathcal A |\, \overline U \cap E_0 \neq \emptyset\},\; E_1:=\{x \in E |\, \exists U \in A_1 \text{ s.t. } x \in \overline U\}, \\ 
 A_{i}&:=\{U \in \mathcal A\setminus \cup^{i-1}_{j=1} A_j |\, \overline U \cap E_{i-1} \neq \emptyset\} \quad (\text{for }i>1),\\
E_i&:=\{x \in E \setminus \cup_{j=1}^{i-1} E_j|\, \exists U \in A_i \text{ s.t. } x \in \overline U\} \quad (\text{for }i>1);
\end{align*}
and we say that such a relatively closed set is {\em admissible}  if 
\[
 \Omega \subset \{x |\, \exists i \in \mathbb N, \,  U \in A_i \text{ s.t. } x \in \overline U\}.
\] 
\end{definition}

\begin{remark}
In particular, we have by definition $E=\bigcup_{i \in \mathbb N} E_i$ for an admissible set $E$.\\
Obviously, every relatively closed set $E \subset \Omega $ with no interior points and such that $\Omega \setminus E$ has only finite many connected components is admissible; e.g., if $E$ has Hausdorff dimension strictly less than $n-1$  because in this case  $\Omega \setminus E$ is connected. \\
A not admissible set is, for example, the union of countable many concentric spheres inside a ball.   
\end{remark}

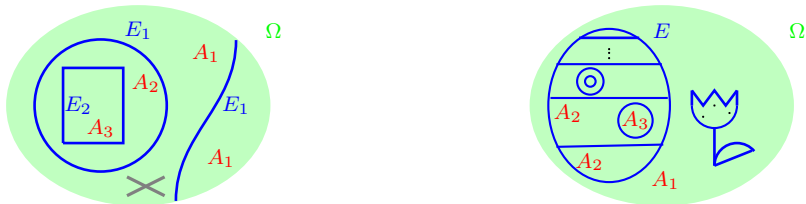
\begin{figure}[htb]
  \centering
  \begin{minipage}[t]{.45\linewidth}
    \centering
     \begin{tikzpicture}
      \fill [green, opacity=0.25] (7,0) ellipse (50pt and 38pt);
      \draw [color=blue] (7,1) node {$\scriptstyle E_{1}$};
      \draw [color=green] (8.8,1) node {$\scriptstyle \Omega$};
      \draw [color=blue] (6.2,0) node {$\scriptstyle E_{2}$};
      \draw [color=blue] (8.3,0) node {$\scriptstyle E_{1}$};
      \draw [color=red] (7.9,0.7) node {$\scriptstyle A_{1}$};
            \draw [color=red] (8.1,-0.7) node {$\scriptstyle A_{1}$};
            \draw [color=red] (7.1,0.3) node {$\scriptstyle A_{2}$};
                  \draw [color=red] (6.5,-0.3) node {$\scriptstyle A_{3}$};
      \draw [color=blue, line width=1pt] (6.5,0) circle (25pt);
      \draw [color=blue, line width=1pt] (6,-0.5)--(6.8,-0.5)--(6.8,0.5)--(6,0.5)--(6,-0.5);
      \draw[color=blue, line width=1pt, out=90,in=-90] (7.5,-1.27) to (8.3,0.87);

\draw[color=gray, line width=1.2pt](6.85,-1.2) - - (7.35,-0.95);
\draw[color=gray, line width=1.2pt](6.85,-0.95) - - (7.35,-1.2);

    \end{tikzpicture}
  \end{minipage}%
  \hfill%
  \begin{minipage}[t]{.45\linewidth}
    \centering
   \begin{tikzpicture}
    \fill [green, opacity=0.25] (7,0) ellipse (50pt and 38pt);
    \draw [color=green] (8.8,1) node {$\scriptstyle \Omega$};
    \draw [color=blue] (7,1) node {$\scriptstyle E$};
    
    \draw [color=blue, line width=0.8pt] (6.3,0) ellipse (23pt and 29pt);
    \draw [color=blue, line width=0.8pt] (5.6,-0.55) -- (7.02,-0.52);
    \draw [color=blue, line width=0.8pt] (5.52,0.1) -- (7.08,0.1);
    \draw [color=blue, line width=0.8pt] (5.6,0.55) -- (7,0.55);
    \draw[fill=blue](6.3,0.75)circle(0.2pt);
    \draw[fill=blue](6.3,0.63)circle(0.2pt);
    \draw[fill=blue](6.3,0.7)circle(0.2pt);
    \draw[color=blue, line width=0.8pt](6.05,0.32) circle(5pt); 
    \draw[color=blue, line width=0.8pt](6.05,0.32) circle(2pt);
                                           \draw[color=blue, line width=0.8pt](6.65,-0.2) circle(6.5pt);
                                                      \draw [color=red] (5.75,-0.1) node {$\scriptstyle A_{2}$};
                                                      
                                                                 \draw [color=red] (6.04,-0.75) node {$\scriptstyle A_{2}$};
                                                                          \draw [color=red] (7.05,-1) node {$\scriptstyle A_{1}$};   
                                                                                                             
                                                                            \draw [color=red] (6.65,-0.2) node {$\scriptstyle A_{3}$};
                                                  
  \draw [>=open triangle 90, color=blue, line width=0.8pt] (7.4,0) arc (180:360:0.3cm); 
\draw[color=blue, line width=1pt] (8,0) -- (8,0.2);
\draw[color=blue, line width=1pt] (7.4,0) -- (7.4,0.2);
\draw[color=blue, line width=1pt] (7.4,0.2) -- (7.55,0) -- (7.7,0.2) -- (7.85,0) -- (8,0.2);
\draw[color=blue, line width=1pt] (7.7,-0.3) -- (7.7,-0.8);
 \draw [>=open triangle 90, color=blue, line width=1pt] (7.7,-0.8) arc (-180:-320:0.3cm);
 \draw[color=blue, line width=1pt] (7.7,-0.8) -- (8.2,-0.6);
 
     \draw[fill=blue](7.9,-0.1)circle(0.2pt);
    \draw[fill=blue](7.55,-0.15)circle(0.2pt);
        \draw[fill=blue](7.7,0)circle(0.2pt);
 
              \draw [color=blue, line width=1pt] (5.9,0.9) -- (6.7,0.9);
  \end{tikzpicture}
  \end{minipage}
  \caption{Two examples of admissible sets}
\end{figure}

\newpage

\begin{definition}\label{defi} \hfill
 \begin{enumerate}
\item We call a sequence $(x_i)_{i \in\mathbb Z} \subset \mathbb R^n$ {\em doubly convergent} to $x$ if 
\[
x_i=x  \text{ iff } i=0   \quad \text{and}   \quad \lim_{i \to \infty} x_i = x = \lim_{i \to - \infty} x_i.
\]
\item A doubly convergent sequence to $x$ is called {\em straight} if the limits exist and
\[ 
x_{-}=\lim_{i \to - \infty} \frac{x_0 -x_i}{\|x_0-x_i\|}=\lim_{i \to  \infty} \frac{x_i -x_0}{\|x_0-x_i\|}=x_+.
\]
\item Given a set $E \subset \R^n$. We say that the sequence $(x_i)_{i \in \mathbb Z} \subset \mathbb R^n$ is {\em semi-transversal} to $E$ at $x \in E$ if
\begin{itemize}
\item $(x_i)_{i \in \mathbb Z}$ is doubly convergent to $x$,
\item there exists a ball $B$ with  $x \in \partial B$, $\overline{B} \cap E = \{x\}$,
\item there exists a closed convex cone $K(x)$ with apex $x$ such that $K(x) \cap B_{\varepsilon}(x) \subset B$ for $\varepsilon>0$ small enough,
\item there exists an $N<0$ with $x_i \in K(x)$ for every $i<N$.
\end{itemize}
Furthermore, if $\Omega \subset \R^n $ is a domain and $E \subset \Omega$ admissible, we say that the sequence $(x_i)_{i \in \mathbb Z} \subset \mathbb R^n$ is {\em outer semi-transversal} to $E$ at $x \in E_i$, $i\geq 1$, if, in addition,
\begin{itemize}
\item $B \subset U$ for one $U \in A_i$. 
 \end{itemize} 
\item A function $u$ defined on the elements of a doubly convergent sequence $(x_i)_{i \in \mathbb Z}$ converging to $x=x_{0}$ is said to be {\em differentiable} with regard to $(x_i)_{i \in \mathbb Z}$ at $x$ if the limits exist and
\[
u_+ =\lim_{i \to \infty} \frac{u(x)-u(x_i)}{\|x-x_i\|}=\lim_{i \to -\infty} \frac{u(x_i)-u(x)}{\|x-x_i\|}=u_-.
\]
\end{enumerate}
\end{definition}

\begin{remark} \label{curvy}
Similarly, one could define the above notions also for a continuous curve.
Then ``straight'' just means that $\gamma$ is  differentiable at $0$ with non-zero tangent vector.\\
We want to  emphasize that $x_i \in E$ for $i\geq 0$ and also  $x_i =x_{-i}$ are allowed in a semi-transversal sequence. In the latter case $u_+=u_-=-u_+=0$ if $u$ is differentiable with regard to $(x_i)_{i \in \mathbb Z}$.
Especially, one can see that a $C^{1}$-function defined in a neighbourhood of $x=x_{0}$ does not need to be differentiable with regard to a doubly convergence sequence unless the sequence is straight, see also Corollary \ref{maincor}.\\
Let us also mention that for $E=\{x\}$ a straight doubly convergent sequence to $x$ is automatically (semi-)transversal to $E$ at $x$ and every doubly convergent sequence to $x$ has  a semi-transversal subsequence.

\end{remark}

Now, we are in the position to state our main theorem. One could vary the assumptions therein a bit to cover other situations. We will briefly indicate these variations afterwards.
 
\begin{theorem} \label{main}
Given a domain $\Omega$ and $E \subset \Omega$ admissible.  Let $u \in C^2(\Omega \setminus E)\cap C(\overline \Omega)$ be an $m$-admissible solution of 
\begin{equation}\tag{\ref{Hessian}}
 F_{m}(D^2u)=S_m(\lambda(D^2u))^{1/m}=f(x) \text{ in } \Omega\setminus E
\end{equation}
with $0< f \in C^{k,\alpha}(\Omega)$, $k \geq 2$, $\alpha \in (0,1)$ such that the Dirichlet problem 
\begin{equation} \tag{\ref{nonsing}}
F_m(D^2v)=f(x) \text{ in } \Omega, \quad  v = \varphi \text{ on } \partial \Omega
\end{equation}
with $\varphi = u_{|\partial \Omega}$ has an $m$-admissible solution $v\in C^2(\Omega)\cap C(\overline \Omega)$ .\\
Then $u$ can be extended to a solution $u \in C^{k+2,\alpha}(\Omega)$ if for every $x \in E$ there exists an outer semi-transversal sequence $(x_i)_{i \in \mathbb Z}$ to $E$ at $x$ such that $u-v$ is differentiable with regard to $(x_i)_{i \in \mathbb Z}$.
\end{theorem}

\begin{proof}
 For sake of clarity we will first assume $\Omega= \bigcup_{A_1} \cup E_1$ with $A_1$ and $E_1$ as in Definition \ref{admissible}, i.e. $E=E_1$ and $\Omega \setminus E=\bigcup_{A_1}$.\\
We prove $u=v$ in $\Omega$ by contradiction, where $v  \in C^{k+2,\alpha}(\Omega)$ follows from the inner regularity, see Remarks \ref{uniformly}.\\
{\bf Step $1$:} Suppose there exists $x \in \Omega$ with $u(x)>v(x)$. Then, the function $w=u-v \in C(\overline \Omega)$ attains its maximum, $\varepsilon >0$, at some point $p\in \overline \Omega$. Write $P:=\{p \in  \overline \Omega|\,w(p)=\varepsilon \}$. Then, 
\[
w(x)=u(x)-v(x) \leq \varepsilon = u(p)-v(p)=w(p)   \text{ for all } x \in \overline \Omega,\, p \in P.
\]
{\bf Step $2$:} Since $ \partial \Omega \cap P= \emptyset$ 
\begin{equation} \label{case}
 \qquad \mathbf{a)} (\Omega \setminus E) \cap P \neq \emptyset \qquad  \text{ or } \qquad   
\mathbf{b)}\;  E \cap P \neq \emptyset.
\end{equation}
We will exclude both cases and thereby achieve $u \leq v$. $u\geq v$ can be shown analogously.\\
{\bf a)} By Proposition \ref{linearelliptic} and Corollary \ref{hesseprop} the linearization $L(u,v)$ is locally uniformly elliptic and the strong maximum principle (Theorem \ref{strong}, here $b\equiv c \equiv 0$) applied to $w$, remind $L(u,v)w=0$, yields that $w \equiv \varepsilon >0$ is constant on every component $U \in A_1$ that contains some $p\in P$. But $w \in C(\overline \Omega)$, $w=0$ on $\partial \Omega$ and $\overline U \cap \partial \Omega \neq \emptyset$ for $U \in A_1$. This forbids case $a)$.\\
{\bf b)} If $ p\in E \cap P$, then there exists by assumption an (outer) semi-transversal sequence $(x_i)_{i \in \mathbb Z}$ to $E$ at $p$ such that $u-v$ is differentiable with regard to $(x_i)_{i \in \mathbb Z}$. In particular, there is a ball $B \subset \Omega \setminus E$ with  $\overline B \cap E = \{p\}$ and a cone $K(p)$ as in the comparison Lemma $\ref{comparison}$ such that $x_i \in B \cap K(p)$ for $i<0$. We can choose $B$ such that $B\subset\subset \Omega$. \\
Set  $v_1 = u_{|\overline B}-\varepsilon \in C^2(B)\cap C(\overline B)$ and $v_0=v_{|\overline B}\in C^{k+2,\alpha}(\overline B)$. Observe that $v_0$, but also $v_1$ is a solution of  
\[
 F_{m}(D^{2}u)=f(x)  \text{ in } B.
\]
Furthermore, $v_0 > v_1$ on $\overline B \setminus \{p\} $ as proved in $a)$ and $v_0(p)=v_1(p)$. $v_0 \in C^{k+2,\alpha}$ gives $v_0 \in C^{3,1}(\overline B)$ as $k\geq2$. Hence, Lemma $\ref{comparison}$ $(iii)$ yields 
\[
\liminf_{i \to - \infty} \frac{(v+ \varepsilon - u)(x_i)}{\Vert x_{i}-p\Vert}=\liminf_{i \to - \infty} \frac{(v_0- v_1)(x_i)}{\Vert x_{i}-p\Vert}>0.
\]
But $\hat w:= v+\varepsilon -u \geq 0$ is by assumption differentiable with regard to $(x_i)_{i \in \mathbb Z}$ and $\hat w$ has a minimum at $p$. This means
\[
0\ge \hat w_+ =\lim_{i \to \infty} \frac{\hat w(p)- \hat w(x_i)}{\|p-x_i\|}=\lim_{i \to -\infty} \frac{\hat w(x_i)-\hat w(p)}{\|p-x_i\|}=\hat w_-\ge 0,
\]
and we arrive at the contradiction
\[
0=\lim_{i \to -\infty} \frac{\hat w(x_i)-\hat w(p)}{\|p-x_i\|}=\liminf_{i \to - \infty} \frac{( v+\varepsilon -u )(x_i)}{\Vert x_{i}-p\Vert} >0.
\]
This finishes the proof under the additional assumption $\Omega= \bigcup_{A_1} \cup E_1$.

Now, we drop this request and allow  $E$ to be an arbitrary admissible set. Then, the proof goes by induction. Step 1 is the same. But instead of only two possibilities for a $p\in P$ to lie in as in \eqref{case} we conclude, since $E$ is admissible:\\
There is an $i \in \mathbb N$ such that there exists  $U \in A_i$ with $U \cap P \neq \emptyset$ or $ E_i \cap  P \neq \emptyset$.\\ 
The case $i=1$ is already treated above. One only has to observe that the ball used in part $b)$ is contained in a component $U \in A_1$ because the sequence $(x_i)_{i \in \mathbb Z}$ is by assumption outer semi-transversal to $E$ at $p$.\\
Assume we have already excluded the cases $j \leq i$. Hence, especially $  E_i \cap P = \emptyset$. Then a similar argument as in $a)$ leads to $  U \cap P = \emptyset$ for $U \in A_{i+1}$ since by definition $\overline U \cap E_i \not = \emptyset$ and we already know that the maximum is not attained at any point of $E_i$.\\
Hence, $ U \cap  P = \emptyset$ for every $U \in A_{i+1}$. But then, by the same arguments as in $b)$, the existence of a $p \in E_{i+1} \cap P$ leads to a contradiction  because of the existence of an outer semi-transversal sequence to $E$ at $p$, i.e. a semi-transversal sequence where the requested ball lies in some $U \in A_{i+1}$.   
\end{proof}

\begin{remark}\label{mainremarks}
The proof can easily be adjusted if we demand only that for every $i \geq 2$ and every $U\in A_i$ there exists at least one $x \in \overline U \cap E_{i-1}$ where the requested semi-transversal sequence is outer semi-transversal. This includes for example a tulip like the one in Figure 1 into our approach. Whereas the existence of a semi-transversal sequence for every $x\in E$, in particular a ball $B$ with  $\overline B \cap E= \{x\}$ seems to be fundamental for our approach. Thus, e.g. a cross as in Figure 1 defies our efforts to remove it.\\
Nevertheless, contemplating on Remarks \ref{curvy} one can see that the general formulation here gives some new insight even in the case of a point singularity; it implies, for example, that if $u=v$ on a sequence converging to the point singularity then $u \equiv v$ on the whole of $\Omega$.\\ 
Of course, one can also apply the theorem if there is an $\Omega' \subset \Omega$ with $E \subset \Omega'$ such that the requirements of the theorem for $\Omega$ are satisfied for $\Omega'$; e.g. if $\Omega' \subset \subset \Omega$ is an $(m-1)$-convex domain with $C^{3,1}$-boundary and $E\subset \subset \Omega'$, then the existence of $v$ is already guaranteed by Theorem \ref{subsec: cns}. It would also be  enough that $E$ could be decomposed such that for every component there exists such an $\Omega'$.\\
Following Remarks \ref{uniformly}, one could weaken the regularity assumptions in case of the Monge-Amp\`ere equation.
\end{remark}

The constraint that $u-v$ has to be differentiable is a bit annoying because it depends on the presumed classical solution $v$. We can remedy the situation by demanding a little bit more.
\begin{corollary} \label{maincor}
With the same assumptions and notations as in Theorem \ref{main} it also holds true that $u$ can be extended to a solution $u \in C^{k+2,\alpha}(\Omega)$ if for every $x \in E$ there exists a straight outer semi-transversal sequence $(x_i)_{i \in \mathbb Z}$ to $E$ at $x$ such that $u$ is differentiable with regard to $(x_i)_{i \in \mathbb Z}$.
\end{corollary}

\begin{proof}
We only have to prove that for a straight doubly convergent sequence $(x_i)_{i \in \mathbb Z}$ a function $v \in C^1(\Omega)$ is automatically differentiable with regard to this sequence. This follows from
\[
v_+=-\nabla v(x) \cdot x_+=\nabla v(x) \cdot (-x_-)=v_-,
\]
using the notations of Definitions \ref{defi}. 
\end{proof}

\begin{remark}
The examples in \cite{C} based on \cite{P} demonstrate that, in general, one has to assume the existence of the classical solution $v$, at least if $E$ is only relatively closed and not compactly contained in $\Omega$. But one can also read it the other way, i.e. that one cannot expect a classical solution for the Monge-Amp\`ere Dirichlet problem if the boundary value $\varphi$ is of regularity less than $C^{1,1-2/n}$, even for a ball and analytic right hand side. So one purpose of this kind of result could be in the negative, i.e. to show nonexistence of classical solutions in a similar way for other equations. More about the existence of non-classical solutions for $m$-Hessian equations can be found in \cite{U}.   
\end{remark}
 
Now, we come to our main example, already mentioned in the introduction, when $E$ is a subset of a $C^{1,1}$-submanifold.
\begin{corollary} \label{intro}
  Let $u\in C^{2}(\Omega\setminus E) \cap  C( \overline \Omega)$ be an $m$-admissible solution of equation \eqref{Hessian} such that the Dirichlet problem \eqref{nonsing}
 has an $m$-admissible solution $v\in C^2(\Omega)\cap C(\overline \Omega)$ with $\varphi=u_{|\partial \Omega}$. Let $E \subset \subset \Omega$ be a closed subset of a $C^{1,1}$-submanifold, $M$, of $\mathbb R^{n}$ of dimension $l<n$. Then $E$ is removable, i.e. $(v=)u\in C^2(\Omega)$, if for every $x \in E$ there is one continuous curve $\gamma$ through $x$, which is differentiable at $x$ and transversal (i.e. not tangential) to $M$ at $x$, such that $u$ is differentiable at $x$ along $\gamma$.
\end{corollary}

\begin{remark}
Instead of $E \subset \subset \Omega$ we could again ask for $E$ just relatively closed in $\Omega$, but then we have to add the assumption of admissibility of $E$. 
\end{remark}

\begin{proof}
We start by showing that $E$ is admissible. To this end, we prove that $\Omega\setminus E$ has only finitely many connected components. Otherwise one could find infinitely many points of $E$ lying on the boundaries of infinitely many different connected components. But then there would exists an accumulation point, $y \in E$,  every whose neighborhood intersecting infinitely many connected components of $\Omega \setminus E$ and therefore also of $\Omega \setminus M$ in contradiction to $y \in M$.\\
Let $x \in E$. By assumption, there is a continuous curve, $\gamma: (-\delta,\delta) \mapsto \R^n$ with $\gamma(0)=x$ such that the derivative, $\gamma'(0)=t_\gamma(x)$, exists at $x$ with $t_\gamma(x) \not \in T_M(x)$ (the tangent space of $M$ at $x$), and $u$ is differentiable at $x$ along $\gamma$. \\
It remains to prove the existence of a straight outer semi-transversal sequence to $E$ at $x$. Project $t_\gamma(x)$ orthogonally to the normal space of $M$ at $x$ and let ${\bf n}_1 \neq 0$ be its image. By Proposition \ref{ball} there exist a ball $B^+$ and a ball $B^-$ such that $\overline B^+ \cap M= \overline B^- \cap M =\{x\}$ and ${\bf n}_1$ is an inner resp. outer normal vector of $\partial B^+$ resp. $\partial B^-$ at $x$. We choose a sequence of the form $(\gamma(t_i))_{i \in \mathbb Z}$ with $t_0=0$, $t_i <0$ for $i<0$ and $t_i >0$ for $i > 0$ or vice versa, and $\lim_{ i \to \infty} t_i =0$, $\lim_{ i \to - \infty} t_i =0$, $t_i$ small enough for all $i \in \mathbb Z$. It is easy to check that such a sequence is indeed straight and outer semi-transversal and that $u$ is differentiable with respect to it since $u$ is differentiable at $x$ along $\gamma$ and $t_\gamma(x)\not =0$. 
\end{proof}

\subsection{General equations}  \label{general}
Let $\Omega$ be a domain and $E\subset \Omega$ admissible. In this subsection we come back to a general function, $F \in C^{k,\alpha}(\Omega, \mathbb R, \mathbb R^n, \mathbb R^{n \times n})$, $k\geq 1$, $0 < \alpha <1$, and the corresponding second order partial differential operators: 
\[
F(x,z,p,r): (\Omega, \mathbb R, \mathbb R^n, \mathbb R^{n \times n}) \mapsto \mathbb R; \quad F[u]=F(x,u,\nabla u, D^2u).
\]
{\bf A)} Our first assumption is that for every open $\Omega' \subseteq \Omega$ there exists a convex set $\mathcal U(\Omega') \subseteq C^2(\Omega')$ for whose elements the operator $F$ becomes elliptic; of course, one could and should always demand that $\mathcal U(\Omega') \subset \mathcal U(\Omega'')$  if $\Omega'' \subset \Omega'$ (with the natural identification by restriction).\\
We call a function $v \in C^2(\Omega)$ $F$-\emph{admissible} in $\Omega$ (in an open subset, $O \subset \Omega$) if $v_{|\Omega'} \in \mathcal U(\Omega')$ for every $\Omega' \subseteq \Omega$ ($\Omega' \subseteq O$).\\[1mm]
Let $u \in C^2(\Omega \setminus E) \cap C(\overline \Omega)$ be an $F$-admissible solution of $F[u]=0$ in $\Omega \setminus E$. Assume also that there is an $F$-admissible solution $v \in C^2(\Omega) \cap C( \overline \Omega)$ of the Dirichlet problem
\[
 F[v]=F(x,v,\nabla v, D^2v)= 0 \text{ in } \Omega, \quad v=u \text{ on } \partial \Omega.
\]
We have $u \in C^{k+2,\alpha}(\Omega \setminus E)$ and $v \in C^{k+2,\alpha}(\Omega)$ by Remark \ref{uniformly}.\\ 
We want to establish our main Theorem in this general setting, too. Scrutinizing the proofs of Theorem \ref{main} and Lemma \ref{comparison} we see that, using Proposition \ref{linearelliptic}, the only additional assumptions we have to impose are: \\[1mm]
 \textbf{B)} $\frac{\partial F(x,z,p,r)}{\partial z}=F_z \leq 0$ and, hence, $F$ is decreasing in $z$.\\[1mm] 
\textbf{C)} For every $a \in E$ there exists an $F$-admissible solution, $u_2 \in C^2(\overline B)$, of the Dirichlet problem
\[
F[u_2]=0 \text{ in } B,\quad u_2=\varphi_{2} \quad \text {for each } \quad \varphi_{2}\in C^{k+2,\alpha}(\partial B),
\]
where $B$ is a ball, requested in the definition of a semi-transversal sequence to $E$ at $a$; of course, we can always choose these balls as small as we want.\\[1mm]
Assumption B is needed to apply the strong maximum and the comparison principle. Because $c \not \equiv 0$ is now possible, one has to observe in the application of the Hopf lemma that $w(a)=0$ in the proof of Lemma \ref{comparison} $(iii)$. At one further instance one has to be a bit cautious; i.e. that the function $v_1 = u_{|\overline B}-\varepsilon$ which appears in the proof of the main Theorem is not necessarily a solution of $F[v_1]=0$ in $B$ when $F(x,u,\nabla u, D^2u)$ depends also on $u$, but with assumption B it is a subsolution, $F[v_1]\geq 0$. Hence, $w=v_1-u_2$ (with $u_2$ as in the proof of Lemma \ref{comparison} $(iii)$) is a subsolution of $L(v_1,u_2) w \geq 0$ and we can apply the comparison lemma as before in the proof of Lemma \ref{comparison}.\\
We want to summarize our analysis of the general case in the following theorem. 

\begin{theorem} \label{generalmain}
Let there be given a domain $\Omega$ and $E \subset \Omega$ admissible. Assume $F \in C^{k,\alpha}(\Omega, \mathbb R, \mathbb R^n, \mathbb R^{n \times n})$, $k\geq 1$, $0 < \alpha <1$, satisfies assumptions \emph{A, B and C}. Let $u \in C^2(\Omega \setminus E)\cap C(\overline \Omega)$ be an $F$-admissible solution of 
\begin{equation}
 F[u]=F(x,u,\nabla u, D^2u)= 0 \text{ in } \Omega\setminus E
\end{equation}
such that the Dirichlet problem 
\begin{equation}
F[v]=0 \text{ in } \Omega, \quad  v = \varphi \text{ on } \partial \Omega
\end{equation}
with $\varphi = u_{|\partial \Omega}$ has an $F$-admissible solution $v\in C^2(\Omega)\cap C(\overline \Omega)$.\\
Then $u$ can be extended to a solution $u \in C^{k+2,\alpha}(\Omega)$ if for every $x \in E$ there exists an outer semi-transversal sequence $(x_i)_{i \in \mathbb Z}$ to $E$ at $x$ such that $u-v$ is differentiable with regard to $(x_i)_{i \in \mathbb Z}$.
\end{theorem}
\begin{remark}
Of course, the remarks we made in the $m$-Hessian case remain valid and, in particular, Corollaries \ref{maincor} and \ref{intro}.\\
Examples can be found, for example, in \cite{CNS} and \cite{T}.
\end{remark}
\appendix 

\section{A Hopf lemma} \label{Hopf}

In this appendix we want to formulate and prove a generalization of the usual Hopf lemma, see \cite{GT} Lemma $3.4.$ and the remark thereafter, which is essential in our proofs.\\
We consider a second order differential operator of the form
\[
Lu = a_{ij}(x)D_{ij}u + b_i(x)D_iu+c(x)u, \quad a_{ij}=a_{ji},
\]
where $x$ lies in a domain $\Omega \subset \R^n$.\\
Let $\lambda(x)$ denote the smallest and $\Lambda(x)$ the greatest eigenvalue of $a_{ij}(x)$.\\ $L$ is {\em elliptic} in $\Omega$ if $\lambda(x)>0$ for every $x \in \Omega$, and {\em uniformly elliptic} if $\frac{\Lambda(x)}{\lambda(x)}$ is bounded in $\Omega$ what is especially the case if  $0<\lambda_0 \leq \lambda(x) \leq \Lambda(x) \leq \Lambda_0$.\\
For a ball $B \subset \R^n$ and a point $x_0 \in \partial B$ we henceforth denote with $K(x_0)$ a closed convex cone with apex $x_0$ such that $K(x_0) \cap B_\varepsilon(x_0) \subset B$ for $\varepsilon >0$ small enough. Such a cone is always contained in a cone with apex $x_0$ and aperture $\pi /2 - \delta$ of the form 
\[
K_\delta(x_0):=\{x|\arccos \frac{\langle x_0-x,x_0-y \rangle}{\|x_0-x\| \|x_0-y \|}\le \frac{\pi}{2}-\delta \},                                                                                                                                                                              
 \] 
where $y$ is the center of the ball $B$ and $0 < \delta \leq \pi/2$, i.e. $K_\delta(x_0)$ is symmetric with regard to the inner normal vector of $\partial B$ at $x_0$.
\begin{lemma}[Hopf lemma] \label{hopf}
Given a ball $B \subset \R^n$, a point $x_0 \in \partial B$ and  a function $u \in C^2(B) \cap C(\overline B)$ such that
\[
 \text{(i)}\,u(x_0) >u(x) \text{ for all }x \in B, \qquad \text{(ii)}\, Lu \geq 0. 
\]
Furthermore, assume $L$ is uniformly elliptic, $\|\mathbf b\|/\lambda$ with $\mathbf b=(b_1,\dots,b_n)$ and $|c|/\lambda$ are bounded in $B$, and $c \equiv 0$, or $c \leq 0$ and $u(x_0)\geq 0$, or $u(x_0)=0$ and $c$ of arbitrary sign. Then
\[
\liminf_{x\to x_{0}, x\in K(x_{0})\cap B}\frac{u(x_{0})-u(x)}{\Vert x-x_{0}\Vert}>0.
 \]
The analogue holds true for the $\limsup$ by reversing all inequalities involving $u$.
 \end{lemma}
First we prove an auxiliary lemma that contains all what is needed to adjust the proof of  
 Lemma $3.4.$ in \cite{GT}.
\begin{lemma} \label{auxhopf}
Given an annular region $R=B_r(0)\setminus \overline B_\rho(0) \subset \R^n$, $r > \rho$, an $\alpha>0$, and a point $x_0 \in \partial B_r(0)$.\\
Assume $u \in C(\overline R) \cap C^2(R)$ and that for an $\varepsilon >0$ holds
\[
u(x_0)-u(x) \geq \varepsilon (e^{-\alpha \|x\|^2}-e^{-\alpha {r}^2}).
\]
Then there exists for every $0 < \delta \leq \frac{\pi}{2} $ an $ \varepsilon'>0$  such that we have for any $x \in K_\delta(x_0) \cap R \cap B_{  \delta'}(x_0)$ with $  \delta' = \frac{1}{2}r \cos (\frac{\pi}{2}-\delta)$
\begin{equation} \label{hopfhilfe}
\frac{u(x_{0})-u(x)}{\Vert x-x_{0}\Vert} \geq \varepsilon' \varepsilon >0.
\end{equation}
\end{lemma}

\begin{proof}
Fix $x \in K_\delta(x_0) \cap R \cap B_{  \delta'}(x_0)$ and define $h(t):=e^{-\alpha\|x_0+t(x-x_0) \|^2}$, $0\leq t \leq 1$. Then there exists $0< \zeta <1$ such that 
\begin{align*}
 e^{-\alpha \|x\|^2}-e^{-\alpha {r}^2}&= h(1)-h(0)=h'(\zeta) \\
      &\hspace{-2.3cm}  = -2\alpha (x-x_{0})\cdot(x_{0}+\zeta(x-x_{0})) e^{-\alpha \|x_{0}+\zeta(x-x_{0})\|^2}\\
 & \hspace{-2.3cm}= 2\alpha e^{-\alpha \|x_{0}+\zeta(x-x_{0})\|^2}\|x-x_0\|( \|x_{0}\| \cos  \measuredangle(x_{0}, x_{0}-x)   -\zeta \|x-x_{0}\|) \\
& \hspace{-2.3cm} \geq 2\alpha e^{-\alpha r^2}\|x-x_0\|( r\cos( \pi/2-\delta) - r/2 \cos (\pi/2-\delta))\\
& \hspace{-2.3cm} \geq \alpha e^{-\alpha r^2}r \cos( \pi/2-\delta)\|x-x_0\|=\varepsilon' \|x-x_0\| >0.
 \end{align*} 
Inequality \eqref{hopfhilfe} follows immediately. 
\end{proof}
Now, we sketch the proof of Lemma \ref{hopf}, following closely the proof of Lemma $3.4.$ in \cite{GT}.
\begin{proof}[Proof of Lemma \ref{hopf}] 
Assume w.l.o.g. $B=B_r(0)$. Take $0 < \rho < r$ and define $v(x):=e^{-\alpha \|x\|^2}-e^{-\alpha {r}^2}$ in $\overline R$ with $R:= B_r(0) \setminus \overline B_\rho(0)$. We compute
\begin{align*}
 (L-c^+)v(x)& = e^{-\alpha \|x\|^2}[4 \alpha^2 a_{ij}x_i x_j -2 \alpha(a_{ii}+b_ix_i)] - c^- v\\
           & \geq e^{-\alpha r^2}[4 \alpha^2 \lambda(x)\rho^2 -2 \alpha(a_{ii}+\|{\bf b}\|r) -c^-] \geq 0
\end{align*}
in $R$ for $\alpha$ large enough because $a_{ii}/\lambda$, $\|\mathbf b\| /\lambda$ and $c/\lambda$ are bounded in $B$. ($c^+:=\max\{0,c\}$, $c^-:=-\min\{0,c\}$)\\
Furthermore, the assumptions imply that $u-u(x_0)+\varepsilon v \leq 0$ on $\partial R$ for an $\varepsilon >0$ and $(L-c^+)(u-u(x_0)+\varepsilon v) \geq - c^+ u+ c^- u(x_0) \geq 0$ in $R$.\\
Therewith, the comparison principle  yields $u-u(x_0)+\varepsilon v \leq 0$ in the whole of $R$ and Lemma \ref{auxhopf} finishes the proof.  
\end{proof}

Because of their importance in our method, we briefly cite the comparison and strong maximum principle from \cite{GT}, Theorem $3.3$ and Theorem $3.5$, compare also the comments after Theorem $3.1$ and Theorem $3.5$ for the assumptions on the coefficients. 
\begin{theorem}[Comparison principle] \label{weak}
Let $u,v \in C^2(\Omega) \cap C(\overline \Omega)$, let $L$ be elliptic with $c\leq 0$ and $\|\mathbf b\|/\lambda$ locally bounded such that $L u \geq L v$ in $\Omega$ and $u \leq v$ on $\partial \Omega$. Then $u \leq v$ in $\Omega$.
\end{theorem}
\begin{theorem}[Strong maximum principle] \label{strong}
Let $u \in C^2(\Omega)$, let $L$ be locally uniformly elliptic with $c\leq 0$ and $\|\mathbf b\|/\lambda$, $|c|/\lambda $ locally bounded such that $L u=0$ in $\Omega$. If $u$ achieves a non-negative maximum or non-positive minimum in the interior of $\Omega$, then $u$ is constant.
\end{theorem}

\section{Existence of a ball} \label{exball}
 
The proof of Corollary \ref{intro} is based on the observation that we can always find suitable balls, touching the points of $E$, if $E$ is a subset of $C^{1,1}$-submanifold. This observation is justified by the following proposition.
\begin{proposition} \label{ball}
Given a $d$-dimensional submanifold $M \subset \mathbb R^n$ of class $C^{1,1}$, a point $p \in M$, a neighborhood $\tilde U_p$ of $p$, an affine hyperplane $\tilde T$ through $p$ with $T_p M \subset \tilde T$ for the tangent space $T_p M$ of $M$, $\dim M=l<n$,  and  ${\bf n}_1$ one of the two unit normal vectors of $\tilde T$ at $p$.\\
Denote by $B_r$ the balls of radius $r$ with $p \in \partial B_r$ and $T_p \partial B_r = \tilde T$ such that ${\bf n}_1$ is the inner normal vector of $\partial B_r$ at $p$.\\
Then there exists an $r_0 > 0$ such that  $\overline B_r \subset \tilde U _p$ and $\overline B_r \cap M = \{p\}$ for all $r \leq r_0$. 
\end{proposition}
\begin{proof}
We assume w.l.o.g. $p=0$ and choose a cartesian coordinate system $(x_1,\dots, x_n)=x_1  \mathbf{ t}_1+\dots + x_l \mathbf{  t}_l+ x_{l+1} \mathbf n_1+ \dots +x_n \mathbf n_{n-l}$ where $\mathbf t_1,\dots, \mathbf t_l$ is an orthonormal basis of $T_p M$ and $ \mathbf n_1,\dots, \mathbf  n_{n-l}$ an orthonormal basis of the normal space $N_p M$ with $\mathbf n_1$ as above and $ \mathbf n_2,\dots, \mathbf  n_{n-l} \in \tilde T$.\\
In a neighborhood $U =U' \times U'' \subset \tilde U_p$ we can parametrize 
\[
M \cap U =\{(x',\varphi(x'))=:M(x')|\,x'\in U'\}
\]
with $\varphi=(\varphi_1,\dots ,\varphi_{n-l}) \in C^{1,1}(U',U'')$, ($U' \subset \mathbb R^{l}, U'' \subset \mathbb R^{n-l}$). We note that with this choice $|\varphi_1|$ is the distance from $M$ to $\tilde T$.\\
Let $\overline r >0$ be so small that $\overline B_{\overline r} \subset U$. Write for $x\in \mathbb R^{n}$
\[
l(x)=l(x_1,\dots,x_l,x_{l+1},\ldots,x_n)=\|(x_1,\dots,x_l,0,x_{l+2},\ldots,x_n)\|_{\mathbb R^n}.
\]
Define for $0 <r \le \overline r$  the cylinder $Z_r$ and the set $V_r \subset U'$  by 
\[
Z_r:=\{x=(x_1, \dots,x_n)|l(x)\le r\}, \quad V_r:=\{x'\in U'|M(x')\in Z_r\}.
\]
Now, fix $0 <r \le \overline r$. For $x'\in V_r$ let 
\[
P(x'):=P(M(x')):=M(x')-\varphi_1(x'){\bf n}_1
\]
 be the orthogonal projection of $M(x')$ on $\tilde T$ and set
\[
R(x'):=P(x')+ \lambda(x')\mathbf  n_1  \text{ with } \lambda(x')=\min \{\lambda \in \mathbb R|P(x')+\lambda \mathbf n_1\in \overline B_r\}.
\]
We have $|\lambda(x')|=r-\sqrt{r^2-l(M(x'))^{2}}$.\\
Moreover, $\overline B_r \cap M=\{p\}$ if $|\varphi_1(x')|<|\lambda(x')|$ for all $x'\in V_r \setminus \{p'\}$.\\ 
Take $x'\in V_r \setminus \{p'\}$. Set $l_1=\|x'\|_{\mathbb R^l}$ and $l_2=l(M(x'))$. Of course, $l_1\le l_2$.\\
Define $h(s):=\varphi_1(s\frac{x'}{l_1})$ for $0\le s\le l_1$. We have $h \in C^{1,1}[0,l_1]$ and
\[
h(l_1)=\varphi_1(x'),\; h(0)=\varphi_1(0)=0,\; h'(s)=\sum_{i=1}^l \frac{\partial \varphi_1(s\frac{x'}{l_1})}{\partial x_i}\frac{x_i}{l_1},\; h'(0)=0.
\]
Furthermore, $h'(s) \leq cs$ if $c$ is the Lipschitz constant of $\nabla \varphi_1$ in $U'$. Since $l_{1}\le r$ we can conclude if $r\leq r_{0}:=\min \{1/c, \overline r\}$
\begin{align*}
 |\varphi_{1}(x')| &= |h(l_1)|=|h(0)+\int_0^{l_1} h'(s)\der s|\le \int_{0}^{l_1}cs \der s =\frac{c}{2}l_1^2  \\ 
                  &\le \frac{1}{2}\frac{l_1^2}{r} < r-\sqrt{r^2-l_1^2}\le r-\sqrt{r^2-l_2^2} = |\lambda(x')|.    
\end{align*}  
\end{proof}

\begin{remark}
Note that this is a generalization of the (interior/exterior) sphere condition, i.e. if $M$ is the boundary of a $C^{1,1}$-domain in $\mathbb R^n$.\\
The regularity $C^{1,1}$ is optimal as the graphs of the $C^{1,\alpha}$-functions $f(t)=|t|^{1+\alpha}$, $0 < \alpha < 1$, show.\\
If one takes the above proposition for an $(n-1)$-dimensional submanifold for granted, it is also possible to prove it by extending $M$ around $p$ locally to an $(n-1)$-dimensional submanifold with the desired unit normal vector. 
\end{remark}

\subsection*{Acknowledgment}
The first author wants to thank her advisor Friedmar Schulz for suggesting the topic, giving helpful remarks and pointing out the references \cite{SW}, \cite{WZ}, \cite{U} and \cite{S}  to her.
This paper will be part of the PhD thesis of the first author.\\
We also thank Anna Dall'Acqua for carefully reading through the paper and providing us with valuable advice and helpful comments.

\end{document}